\documentclass[a4paper]{article}
\usepackage[T1]{fontenc} 
\usepackage[utf8]{inputenc} 
\usepackage[english]{babel} 
\usepackage{alphabeta} 
\usepackage{authblk}
\usepackage[frozencache=true,cachedir=minted-cache]{minted}
\usepackage{amsmath, amsthm, amsfonts}
\usepackage{enumerate}
\usepackage{hyperref}
\usepackage{amssymb}
\usepackage{graphicx}
\usepackage[all]{xy}
\usepackage{tikz-cd}
\usepackage{orcidlink}
\usetikzlibrary{graphs,decorations.pathmorphing,decorations.markings}

\newtheorem{thm}{Theorem}[]
\newtheorem{cor}[thm]{Corollary}
\newtheorem{lem}[thm]{Lemma}
\newtheorem{prop}[thm]{Proposition}
\newtheorem{exam}[thm]{Example}
\theoremstyle{definition}
\newtheorem{defi}[thm]{Definition}
\theoremstyle{remark}

\newtheorem{rem}[thm]{Remark}

\newcommand*{\cat}[1]{\mathcal{#1}}
\newcommand*{\set}{\operatorname*{set}}
\newcommand*{\obj}[1]{#1_0}

\newcommand*{\Hom}[3]{\operatorname{Hom}_{#1} \left(#2, #3\right)}

\newcommand*{\op}[1]{#1^{op}}
\newcommand*{\CSet}[1]{#1-{\set}}
\newcommand*{\orbit}[3]{\cat{#1}_w^{#2} \cdot #3}
\newcommand*{\gen}[3]{\cat{#1}^{#2} \cdot #3}
\newcommand*{\biset}[3]{{}_{#1}{#2}_{#3}}
\newcommand*{\comp}[6]{\biset{#1}{#2}{#3} \times_{#3} \biset{#4}{#5}{#6}}

\newcommand*{\ZZ}{\mathbb{Z}}
\newcommand*{\NN}{\mathbb{N}}

\title{Characterization of groupoid categories in terms of its category of $\cat{C}$-sets}

\author[1]{J. Miguel Calder\'on \orcidlink{0009-0007-4685-8837}}
\author[2]{Alberto G. Raggi-C\'ardenas \orcidlink{0000-0003-1720-1733}}
\author[3]{Itzel Rosas \orcidlink{0009-0002-3411-8456}}
\author[4]{Ram\'on H. Ruiz-Medina \orcidlink{0000-0003-2916-9160}}
\date{}

\affil[1]{Centro de Ciencias Matem\'aticas.
	Universidad Nacional Aut\'onoma de M\'exico.
	Morelia, Michoac\'an 58089  \\ \texttt{calderonl@matmor.unam.mx}}
\affil[2]{Centro de Ciencias Matem\'aticas.
	Universidad Nacional Aut\'onoma de M\'exico.
	Morelia, Michoac\'an 58089  \\ \texttt{agraggi@gmail.com}}
\affil[3]{PCCM, UNAM -- UMSNH.
	Morelia, Michoac\'an\\ \texttt{irosas@matmor.unam.mx}}
\affil[4]{Centro de Ciencias Matem\'aticas.
	Universidad Nacional Aut\'onoma de M\'exico.
	Morelia, Michoac\'an 58089  \\ \texttt{harath@matmor.unam.mx}}

\begin{document}

\maketitle

\begin{abstract}
    A $\cat{C}$-set is a functor from the category $\cat{C}$ to the category of finite sets and functions. The category of $\cat{C}$-sets, $\CSet{\cat{C}}$, is defined as the category whose objects are $\cat{C}$-sets, and whose morphisms are natural transformations between them. In this document we provide some concise characterizations of groupoids in terms of their category of $\cat{C}$-sets. \\
	\textbf{Keywords:} Groupoid, $\cat{C}$-set, indecomposable object, Burnside ring. \\
	\textbf{MSC 2020:} 18A25, 18B40, 19A22.
\end{abstract}


\section{Introduction}

%
%
%

Group actions are a fundamental tool in understanding algebraic structures, with applications in symmetry, combinatorics, and representation theory. Classical results --such as Burnside's lemma and the study of the Burnside ring \cite{tomDieck}-- illustrate how actions encode essential information about finite groups. This framework also appears in diverse areas, including geometric invariants and permutation representations \cite{DummitFoote}. However, as mathematical structures increase in complexity, so does the need for generalizations that preserve their underlying categorical nature.

Category theory provides a natural setting for such generalizations. By formalizing structure-preserving transformations, categories extend the concept of groups: morphisms replace group elements, and composition generalizes multiplication \cite{MacLane}. This shift in perspective has proven useful in various areas, including algebraic topology, homological algebra, and even theoretical computer science \cite{Simmons}. In particular, the transition from groups to categories invites a parallel extension of actions: just as groups act on sets, categories can ``act'' via functors to sets, providing a broader framework for analyzing structure.

In this work, we explore this idea by generalizing group actions to category actions, building on the approach introduced by P. Webb \cite{Webb23} in the context of biset functors. We focus instead on the internal properties of category actions themselves. In this setting, we consider $\cat{C}$-sets --functors from a category $\cat{C}$ to $\set$-- as a natural extension of the notion of $G$-sets. These functors can be thought of as category actions. As shown in \cite{Rosas}, this framework enables richer descriptions of certain categorical structures --particularly groupoids-- through the behavior of their associated $\cat{C}$-sets. Here, we provide a characterization of groupoids in terms of their categories of $\cat{C}$-sets, bridging combinatorial and categorical viewpoints.

The main objective of this article is to prove that, if $\cat{C}$ is a category of finite type --that is, the number of indecomposable $\cat{C}$-sets, up to isomorphism, is finite-- and, in addition, $\cat{C}$ is finite and connected, then the following theorem holds. 

\begin{thm}\label{t1}
	Let $\cat{C}$ be a finite, connected category. Then the following statements are equivalent:
	\begin{enumerate}
		\item[i)] $\cat{C}$ is semisimple.
		\item[ii)] $\cat{C}$ is of finite type.
		\item[iii)] $\cat{C}$ is a groupoid.
		\item[iv)] For all $\cat{C}$-set $\Omega: \cat{C} \longrightarrow \set$, all object $x \in \obj{\cat{C}}$, and all element $a \in \Omega(x)$, we have: \[
		\cat{C}^{\Omega}\cdot a = \cat{C}_{w}^{\Omega}\cdot a.
		\]
	\end{enumerate}
\end{thm}


\section{Preliminaries}

Throughout this work, we consider finite categories --those whose classes of objects and morphisms are finite sets.\\

\begin{defi}
	Let $\cat{C}$ be a category. A \textit{$\cat{C}$-set} is a covariant functor from the category $\cat{C}$ to the category $\set$ of finite sets and functions.
\end{defi}

For notational convenience, rather than writing $\Omega(\alpha)(u)$ to denote the image of an element $u \in \Omega(x)$ under the induced map, we abbreviate it as $\alpha u$.

A $\cat{C}$-set $\Omega: \cat{C} \longrightarrow \set$ is \textit{finite} if, for all $x \in \obj{\cat{C}}$, the set $\Omega(x)$ is finite. The $\cat{C}$-sets form a category $\CSet{\cat{C}}$, whose objects are finite $\cat{C}$-sets and whose morphisms are natural transformations between them.

Note that when $\cat{C}$ is the delooping of a group $G$ (that is, a category with a single object whose morphisms are the elements of the group), the category of $\cat{C}$-sets is isomorphic to the category of $G$-sets.

\begin{exam}
	Let $\cat{D}$ be the following quiver:
	\begin{center}
		\begin{tikzcd}
			\cat{D} = & x \arrow{r}{\alpha} & y
		\end{tikzcd}
	\end{center}
	and consider it as category.
	
	A $\cat{D}$-set $\Omega: \cat{D} \longrightarrow \set$ can be seen as a triple $\Omega = (X, Y, f)$ where $X = \Omega(x)$, $Y = \Omega(y)$ and $f: X \longrightarrow Y$ is such that $f = \Omega(\alpha)$.
\end{exam}

Let $\Omega$ be a $\cat{C}$-set. We say that $\Omega_1$ is a \textit{$\cat{C}$-subset} of $\Omega$ if $\Omega_1$ is a subfunctor of $\Omega$. We have the following characterization:

\begin{prop}[\cite{Rosas}, Proposición 1.14, (1)]
	Let $\Omega$ be a $\cat{C}$-set. $\Omega_1$ is a $\cat{C}$-subset of $\Omega$ if and only if \[
		\Omega_1(x) \subseteq \Omega(x)
	\] for every object $x$ of $\cat{C}$ and, for every $\alpha \in \Hom{\cat{C}}{x}{y}$, \[
		\Omega(\alpha) \left(\Omega_1(x)\right) \subseteq \Omega_1(y).
	\]
\end{prop}

In some texts, this proposition is taken as the definition of a $\cat{C}$-subset, as in the case in Definition 2.2 of \cite{Webb23}.

Given two $\cat{C}$-sets $\Omega_1$ and $\Omega_2$ we define their \textit{disjoint union} $\Omega_1 \sqcup \Omega_2$ to be the $\cat{C}$-set defined at each object $x \in \obj{\cat{C}}$ by \[
	(\Omega_1 \sqcup \Omega_2)(x) := \Omega_1(x) \sqcup \Omega_2(x)
\] and on morphisms $\alpha: x \longrightarrow y$ in $\cat{C}$ \[
	(\Omega_1 \sqcup \Omega_2)(\alpha) := \Omega_1(\alpha) \sqcup \Omega_2(\alpha)
\] sends each element of $\Omega_i(x)$ to the corresponding one in $\Omega_i(y)$.

We say that a $\cat{C}$-set $\Omega$ is \textit{non-empty} if exists $c \in \obj{\cat{C}}$ such that $\Omega(c) \neq \emptyset$. We will denote the empty $\cat{C}$-set as $\varnothing$.

\begin{defi}
	A non-empty $\cat{C}$-set $\Omega$ is said to be \textit{indecomposable} if it cannot be expressed properly as a disjoint union, that is, if $\Omega_1, \Omega_2$ are $\cat{C}$-sets such that $\Omega = \Omega_1 \sqcup \Omega_2$, then $\Omega_1 = \varnothing$ or $\Omega_2 = \varnothing$.
\end{defi}

As with $G$-sets, any $\cat{C}$-set can be decomposed as a disjoint union of indecomposable sets, as shown in the next result:

\begin{thm}[\cite{Webb23}, Proposition 2.3, (2)]
	Let $\cat{C}$ be a category. Up to order, every non-empty finite $\cat{C}$-set $\Omega$ has a unique decomposition \[
		\Omega = \Omega_1 \sqcup \Omega_2 \sqcup \cdots \sqcup \Omega_n
	\] where each $\Omega_i$ is indecomposable.
\end{thm}

\begin{exam} \label{Omega_n}
	Consider the category 
	\begin{center}
		\begin{tikzcd}
			\cat{D} = & x \arrow{r}{\alpha} & y
		\end{tikzcd}
	\end{center}
	
	The indecomposable finite $\cat{D}$-sets are isomorphic to the $\cat{D}$-sets of the form \begin{center}
		\begin{tikzcd}
			\Omega_n = & \underline{n} \arrow{r}{c_1} & \underline{1}
		\end{tikzcd}
	\end{center}
	where for every $n \in \NN$, $\underline{n} = \{1, \ldots, n\}$, and $c_1$ is the 1 constant function. Note that $\Omega_0$ is defined as follows: $\Omega_0(x) = \emptyset$, $\Omega_0(y) = \{1\}$ and $\Omega_0(\alpha)$ is the only function from $\emptyset$ to $\{1\}$.
	
%
%
\end{exam}

\begin{exam}\label{Omega_n2}
	Consider the category 
	\[
	\cat{D} :=
	\begin{tikzcd}
		x \arrow[out=-30,in=30,loop, swap]{r}{\alpha}
	\end{tikzcd}
	\]
	
	The indecomposable finite $\cat{D}$-sets are isomorphic to the $\cat{D}$-sets of the form 
	\[
	\Omega_n :=
	\begin{tikzcd}
		\underline{n} \arrow[out=-30,in=30,loop, swap]{r}{f}
	\end{tikzcd}
	\]
	for every $n \in \NN$, where $f$ is a function with the following property: if $Y \subseteq \underline{n}$ is a maximal subset of $\underline{n}$ such that $f|_Y: Y \longrightarrow Y$ is a permutation, then $f|_Y$ is an only cycle. Note that $\Omega_0$ is the empty $\cat{D}$-set.
	
\end{exam}

A category $\cat{C}$ is said to be of \textit{finite type} if the number of indecomposable $\cat{C}$-sets, up to isomorphism, is finite.

A useful tool for a deeper understanding of the indecomposable $\cat{C}$-sets is presented below.

\begin{defi}
	Let $\Omega: \cat{C} \longrightarrow \set$ be a $\cat{C}$-set and let $x, y \in \obj{\cat{C}}$. A \textit{walk} from $x$ to $y$ is a sequence of morphisms $\alpha = \alpha_n \cdots \alpha_1$ of the form \begin{center}
		\begin{tikzcd}
			x \arrow[dash]{r}{\alpha_1} & y_1 \arrow[dash]{r}{\alpha_2} & y_2 \arrow[dash]{r} & \cdots  \arrow[dash]{r}{\alpha_n} & y_n = y
		\end{tikzcd}
	\end{center} 
	where \begin{tikzcd}[cramped] y_{i - 1} \arrow{r}{\alpha_i} & y_i \end{tikzcd} or \begin{tikzcd}[cramped] y_{i - 1} & y_i \arrow[swap]{l}{\alpha_i} \end{tikzcd} for all $i$. We denote a walk $\alpha$ from $x$ to $y$ by $\alpha: x \rightsquigarrow y$ or \begin{tikzcd}[cramped] x \arrow[rightsquigarrow]{r}{\alpha} & y \end{tikzcd}.
	
	A walk is said to be \textit{reduced} if no two consecutive morphisms $\alpha_i$ and $\alpha_{i + 1}$ of the walk can be composed.
	
	If $\alpha$ is a walk from $x$ to $y$, we define the \textit{opposite walk} of $\alpha$, denoted by $\op{\alpha}$, as the walk going on the opposite direction, that is, $\op{\alpha}$ is the walk from $y$ to $x$.
\end{defi}


Let $x, y \in \obj{\cat{C}}$ and $u \in \Omega(x)$. With this, we define the \textit{$\cat{C}$-orbit of $u$ with respect to walks, evaluated at $y$}, as:
\[
	\left(\orbit{C}{\Omega}{u}\right)(y) = \bigcup_{\substack{x \rightsquigarrow y \\ \alpha}} \alpha u
\]
where $\alpha$ is a walk from $x$ to $y$, and $\alpha u$ denotes its action on $u$.

We have a characterization of indecomposable $\cat{C}$-sets in terms of orbits with respect to walks:

\begin{thm}[\cite{Rosas}, Teorema 1.28]
	Let $\Omega$ be a $\cat{C}$-set. The following conditions are equivalent:
	\begin{enumerate}
		\item $\Omega$ is an indecomposable $\cat{C}$-set;
		
		\item For every object $x$ of $\cat{C}$ and every element $u \in \Omega(x)$, $\orbit{C}{\Omega}{u} = \Omega$;
		
		\item There exists an object $x$ of $\cat{C}$ and an element $u \in \Omega(x)$ such that $\orbit{C}{\Omega}{u} = \Omega$.
	\end{enumerate}
\end{thm}

A category $\cat{C}$ is said to be \textit{connected} if, for every $x, y \in \obj{\cat{C}}$, there exists a walk between $x$ and $y$. Equivalently, a category $\cat{C}$ is said to be connected if it cannot be expressed as $\cat{C} = \cat{C}^1 \sqcup \cat{C}^2$, where $\cat{C}^1$ and $\cat{C}^2$ are non-empty full subcategories of $\cat{C}$ such that no morphism start in one and finish in the other one.

\begin{defi}
	Let $\Omega$ be a $\cat{C}$-set, let $x, y \in \obj{\cat{C}}$ and $u \in \Omega(x)$. The $\cat{C}$-set \textit{generated by} $u$, \textit{evaluated at} $y$ is given by  
	\begin{align*}
		(\gen{C}{\Omega}{u})(y) = \Omega \left(\Hom{\cat{C}}{x}{y}\right)(u)
	\end{align*}
\end{defi}

In a natural way, given $x \in \obj{\cat{C}}$ and $u \in \Omega(x)$, we have that $\gen{C}{\Omega}{u}$ is a $\cat{C}$-subset of $\Omega$.

\begin{rem}
	Let $\Omega$ be a $\cat{C}$-set, and let $x, y \in \obj{\cat{C}}$ and $u \in \Omega(x)$. We have that $\gen{C}{\Omega}{u}$ is a $\cat{C}$-subset of $\orbit{C}{\Omega}{u}$.
\end{rem}

Note that the equality does not always hold. We now present an example where it fails:

\begin{exam}
	Consider the category 
	\begin{center}
		\begin{tikzcd}
			\cat{D} = & x \arrow{r}{\alpha} & y
		\end{tikzcd}
	\end{center} 
	and let $\Omega_n$ be as in Example \ref{Omega_n}. Consider the element $1 \in \Omega_n(y)$. We have that $(\gen{D}{\Omega_n}{1})(x) = \emptyset$, since $\Hom{\cat{D}}{y}{x} = \emptyset$.
	
	On the other hand, the preimage $\Omega_n(\alpha)^{-1}(1) = \underline{n} \subseteq (\orbit{D}{\Omega_n}{1})(x)$. We conclude that $(\orbit{D}{\Omega_n}{1})(x) = \underline{n}$. This means that $\gen{D}{\Omega_n}{1} \neq \orbit{D}{\Omega_n}{1}$.
\end{exam}

\begin{defi}
	A non-empty $\cat{C}$-set $\Omega$ is said to be \textit{simple} if it has no subfunctors apart from the whole functor and the empty functor.
\end{defi}

\begin{exam}
	Let $\cat{D}$ be the category defined in Example \ref{Omega_n}, and let $\Omega_n$ be the indecomposable $\cat{D}$-sets as defined therein. Note that $\Omega_0$ is the only $\cat{D}$-simple.
\end{exam}

A category $\cat{C}$ is \textit{semisimple} if every $\cat{C}$-set is isomorphic to a finite disjoint union of simple $\cat{C}$-sets.

Note that every simple $\cat{C}$-set is indecomposable, but not every indecomposable $\cat{C}$-set is simple, as we can see in Example \ref{Omega_n}. In that case, the only non-empty simple $\cat{D}$-set is $\Omega_1$. However, a category $\cat{C}$ is semisimple if and only if every indecomposable $\cat{C}$-set is simple. 

As in the case of indecomposable sets, we have a characterization of simple $\cat{C}$-sets in terms of generated $\cat{C}$-sets:

\begin{thm} \label{simples}
	Let $\Omega$ be a $\cat{C}$-set. The following conditions are equivalent:
	\begin{enumerate}
		\item $\Omega$ is a simple $\cat{C}$-set;
		
		\item For every object $x$ of $\cat{C}$ and every element $u \in \Omega(x)$, $\gen{C}{\Omega}{u} = \Omega$.
	\end{enumerate}
\end{thm}

\begin{defi}
	We define the \textit{Burnside ring} of the category $\cat{C}$, denoted by $B(\cat{C})$, as the Grothendieck group of the category $\CSet{\cat{C}}$ with respect to the disjoint union of $\cat{C}$-sets.
\end{defi}

	We have a product operation in $\CSet{\cat{C}}$: given two $\cat{C}$-sets $\Omega, \Psi$, we define $(\Omega \times \Psi)(x) = \Omega(x) \times \Psi(x)$ for every object $x$ of $\cat{C}$, and $(\Omega \times \Psi)(\alpha) = (\Omega(\alpha), \Psi(\alpha))$ for every morphism $\alpha$ of $\cat{C}$. With this definition, the multiplication in $B(C)$ is given by $\times$ on the generators.

As in the case of finite groups, we have the following theorem:

\begin{thm}[\cite{Webb23}, Proposition 2.9]
	Let $\cat{C}$ be a category. Its Burnside ring $B(\cat{C})$ is a commutative ring with identity, and with $\ZZ$-basis the set of isomorphism classes of the indecomposable $\cat{C}$-sets.
\end{thm}

While bisets can themselves be seen as a generalization of $G$-sets, we further extend this notion to the categorical setting. A more general framework for this theory can be found in \cite[Section 3]{Webb23} and \cite[Section 7.8]{Borceux}.

\begin{defi}
	Given categories $\cat{C}$ and $\cat{D}$, we define a $(\cat{C}, \cat{D})$\textit{-biset} to be a functor $\Omega: \cat{C} \times \op{\cat{D}} \longrightarrow \set$, that is, a $\cat{C} \times \op{\cat{D}}$-set. We will denote it by $\biset{\cat{C}}{\Omega}{\cat{D}}$.
\end{defi}

\begin{rem}
	A $\cat{C}$-set $\Omega$ can be regarded as a $(\cat{C}, \mathbf{1})$-biset, where $\mathbf{1}$ is the category with a single object and a single morphism. We denote such a $\cat{C}$-set as $\biset{\cat{C}}{\Omega}{}$ and, when the category is clear from context, simply by $\Omega$.
\end{rem}

Given a $(\cat{C}, \cat{D})$-biset $\Omega$, morphisms $\alpha: x \longrightarrow x_1$ in $\cat{C}$ and $\beta: y_1 \longrightarrow y$ in $\cat{D}$, and an element $u \in \Omega(x, y)$, we get elements \begin{eqnarray*}
	\alpha u := \Omega(\alpha, 1_y)(u) \in \Omega(x_1, y), \\
	u \beta := \Omega(1_x, \beta)(u) \in \Omega(x, y_1).
\end{eqnarray*}

If $\alpha$ and $\beta$ are walks, we denote the image of $u$ under the corresponding walk using the same notation as for morphisms.

\begin{exam}
	If $\cat{C}$ is the delooping of a finite group $G$ and $\cat{D}$ is the delooping of a finite group $H$, a $(\cat{C}, \cat{D})$-biset is a $(G, H)$-biset.
\end{exam}

We can form disjoint unions of bisets, but there is also a notion of \textit{product} or \textit{composition}, which is described below.

\begin{defi}
	Given a $(\cat{C}, \cat{D})$-biset $\biset{\cat{C}}{\Omega}{\cat{D}}$ and a $(\cat{D}, \cat{E})$-biset $\biset{\cat{D}}{\Psi}{\cat{E}}$ we build a $(\cat{C}, \cat{E})$-biset $\Omega \times_{\cat{D}} \Psi$ by the formula, for $x \in \obj{\cat{C}}$ and $z \in \obj{\cat{E}}$, \[
	(\Omega \times_\cat{D} \Psi)(x, z) = \left(\bigsqcup_{y \in \obj{\cat{D}}} \Omega(x, y) \times \Psi(y, z)\right) / \sim, 
	\] where $\sim$ is defined as follows.
	Let $(u, b) \in \Omega(x, y) \times \Psi(y, z)$ and $(a, v) \in \Omega(x, y') \times \Psi(y', z)$, we say that $(u, b) \sim (a, v)$ if and only if there exists a walk $\beta: y \rightsquigarrow y'$ in $\cat{D}$ such that $u \in a \beta$ and $v \in \beta b$.
\end{defi}

Given categories $\cat{C}$, $\cat{D}$ and $\cat{E}$, and functors $F: \cat{C} \longrightarrow \cat{E}$ and $G: \cat{D} \longrightarrow \cat{E}$ we obtain a $(\cat{C}, \cat{D})$-biset that we denote $\biset{\cat{C}^F}{\cat{E}}{{}^G\cat{D}}$. In objects $x$ of $\cat{C}$ and $y$ of $\cat{D}$ this biset is defined by \[
	\biset{\cat{C}^F}{\cat{E}}{{}^G\cat{D}}(x, y) = \Hom{\cat{E}}{G(y)}{F(x)}.
\]

For example, if we take $\cat{C} = \cat{D} = \cat{E}$ and $F = G = 1_{\cat{C}}$ we get the biset $\biset{\cat{C}}{\cat{C}}{\cat{C}}$ defined in objects as $\biset{\cat{C}}{\cat{C}}{\cat{C}}(x, y) = \Hom{\cat{C}}{y}{x}$.

As in biset theory for finite groups, if $\cat{D}$ is a subcategory of $\cat{C}$ and we consider the inclusion functor $\iota: \cat{D} \longrightarrow \cat{C}$, then the biset $\biset{\cat{D}^\iota}{\cat{C}}{\cat{C}} = \biset{\cat{D}}{\cat{C}}{\cat{C}}$ encodes the restriction operation, and the biset $\biset{\cat{C}}{\cat{C}}{{}^\iota \cat{D}} = \biset{\cat{C}}{\cat{C}}{\cat{D}}$ encodes the induction operation.


\section{Proof of the characterization theorem}

This section presents the main result of this work: a characterization of groupoids in terms of their simple and indecomposable $\cat{C}$-sets. The proof relies on a sequence of technical propositions that will be used to establish the main theorem. In this section, we will work only with connected categories, since the Burnside ring of a disconnected category can be seen as the product of the Burnside rings of its connected components.

\begin{thm} \label{adj}
	Let $\cat{D}$ be a subcategory of a category $\cat{C}$, and let $\biset{\cat{D}}{\cat{C}}{\cat{C}}$ and $\biset{\cat{C}}{\cat{C}}{\cat{D}}$ be the restriction and induction operations defined previously. Consider the functors \[
		\begin{tikzcd}
			\CSet{\cat{D}} \arrow[shift left]{r}{i} & \CSet{\cat{C}} \arrow[shift left]{l}{r}
		\end{tikzcd}
	\] given as follows: \begin{itemize}
		\item Let $\Psi \in \obj{(\CSet{\cat{C}})}$, then $r(\Psi) = \comp{\cat{D}}{\cat{C}}{\cat{C}}{}{\Psi}{}$, 
		
		\item Let $\Omega \in \obj{(\CSet{\cat{D}})}$, then $i(\Omega) = \comp{\cat{C}}{\cat{C}}{\cat{D}}{}{\Omega}{}$.
	\end{itemize}
	
	Then $r$ and $i$ are adjoint functors.
\end{thm}

\begin{proof}
	We give the explicit isomorphisms:
	
	\[
		\begin{tikzcd}
			\Hom{\CSet{\cat{C}}}{i(\Omega)}{\Psi} \arrow[shift left]{rr}{\alpha} && \Hom{\CSet{\cat{D}}}{\Omega}{r(\Psi)} \arrow[shift left]{ll}{\beta}
		\end{tikzcd}
	\]
	
	Let $x \in \obj{\cat{D}}$, $y \in \obj{\cat{C}}$ be objects, and let $\zeta: i(\Omega) \longrightarrow \Psi$ and $\theta: \Omega \longrightarrow r(\Psi)$ be morphisms in the corresponding categories.
	
	For $\alpha$: note that we can see $r(\Psi)$ as $\Psi$, but restricted to the objects in $\cat{D}$ only. Then $\alpha(\zeta)_x : \Omega(x) \longrightarrow r(\Psi)(x) = \Psi(x)$, so if $a \in \Omega(x)$, \[\alpha(\zeta)_x(a) = \zeta_x([1_x, a]).\]
	
	For $\beta$: recall that $i(\Omega) = \comp{\cat{C}}{\cat{C}}{\cat{D}}{}{\Omega}{}$. Then $\beta(\theta)_y : i(\Omega)(y) \longrightarrow \Psi(y)$, so if $[f, b] \in i(\Omega)(y)$, \[
		\beta(\theta)_y([f, b]) = \Psi(f)(\theta_z(b)),
	\] with $z \in \obj{\cat{D}}$ such that $b \in \Omega(z)$.
	
	Note that both $\alpha$ and $\beta$ are well-defined and are inverses of each other.
\end{proof}

\begin{cor} \label{adj2}
	There exists, for every $\cat{D}$-set $\Omega$, with $\cat{D}$ a subcategory of a category $\cat{C}$, a natural transformation \[
		\begin{tikzcd}
			\eta: \Omega \arrow{r} & ri(\Omega)
		\end{tikzcd}
	\] such that for every $y \in \obj{\cat{D}}$, $\eta_y(a) = [1_y, a]$.
\end{cor}

\begin{proof}
	This natural transformation is the unit of the adjunction given in Theorem \ref{adj}.
\end{proof}

%
%

\begin{lem} \label{lemw}
	Let $x$ and $y$ be objects of a category $\cat{C}$, and let 
	\begin{tikzcd}[cramped] 	
		\cat{D} = x \arrow{r}{\alpha} & y
	\end{tikzcd} 
	be a subcategory of $\cat{C}$ such that $\alpha$ is not an isomorphism in $\cat{C}$. Without loss of generality, suppose that $\alpha$ is not a retraction in $\cat{C}$. Define the walk $\beta := (\op{\alpha} \alpha)^{t}$, for some $t \geq 1$. Then $1_{x} \cdot \beta = \emptyset$.
\end{lem}

\begin{proof}
	Suppose that $1_{x} \cdot \beta \neq \emptyset$. Without loss of generality, this implies that $1_{x} \cdot (\op{\alpha} \alpha) \neq \emptyset$, so we take an element $l \in 1_{x} \cdot (\op{\alpha} \alpha) = (\op{\alpha} \alpha)^{*}(1_{x}) = \alpha^{*}((\op{\alpha})^{*}(1_{x}))$. Since $(\op{\alpha})^{*}(1_{x})$ is a nonempty set, there exists $\lambda \in (\op{\alpha})^{*}(1_{x})$, so that $\alpha^{*}(\lambda) = \lambda \alpha = 1_{x}$, which contradicts the fact that $\alpha$ is not a retraction.
\end{proof}

\begin{lem}
	Let $x$ and $y$ be objects of a category $\cat{C}$, and let $\alpha: x \longrightarrow y$ be a morphism in $\cat{C}$ such that it has no left inverse. Consider the subcategory \begin{tikzcd}[cramped] \cat{D} = x \arrow{r}{\alpha} & y \end{tikzcd} of $\cat{C}$, and the $\cat{D}$-sets $\Omega_n$ defined as in Example \ref{Omega_n} or Example~\ref{Omega_n2}. Then there exists a natural transformation from $\Omega_n$ to $ri(\Omega_n)$, which is injective on each component. 
\end{lem}

\begin{proof}
%
%
%

	Corollary \ref{adj2} gives us the existence of this natural transformation which, indeed, is the unit of the adjunction given in Theorem \ref{adj}. We call it $\eta$. It is enough to prove that $\eta_x$ is injective.
	
	Assume that $\eta_x(a) = \eta_x(b)$, which means that $[1_x, a] = [1_x, b]$. By the definition of the equivalence relation in the biset composition, there exists a walk $\beta$ in $\cat{D}$ from $x$ to $x$ such that $1_x \in 1_x \cdot \beta$ and $b \in \beta \cdot a$.
	
	Given the structure of the category $\cat{D}$, there are only two possibilities for the walk $\beta$: either $\beta$ is the identity morphism or it is of the form $(\op{\alpha}\alpha)^t$ for some $t \geq 1$. Suppose $\beta$ is not the identity; then, by Lemma~\ref{lemw}, $\beta$ must be one of the walks described therein, and in particular, $1_x \cdot \beta = \emptyset$, which is a contradiction.
	
	Therefore, $\beta$ must be the identity morphism. In this case, $b \in \beta \cdot a$ implies $b = \beta \cdot a = \mathrm{id} \cdot a = a$. Hence, $\eta_x$ is injective.
\end{proof}

We now proceed with the proof of the Theorem \ref{t1}.

\begin{proof}
	\begin{itemize}
		\item[] $i) \implies ii)$ As $\cat{C}$ is a finite category, there exists a finite number of simple $\cat{C}$-sets, up to isomorphism. As in a semisimple category any $\cat{C}$-set is simple if and only if it is indecomposable, this implies that there is a finite number of indecomposable objects. Hence, $\cat{C}$ is of finite type.
		
		\item[]  $ii) \implies iii)$ 
		Suppose that $\cat{C}$ is of finite type. Consider two objects $x, y \in \obj{\cat{C}}$ and a morphism $\alpha: x \longrightarrow y$ such that is not an isomorphism. Assume, without loss of generality, that $\alpha$ has no left inverse. We then define the subcategory \begin{tikzcd}[cramped] 	
			\cat{D} = x \arrow{r}{\alpha} & y
		\end{tikzcd} and let $\Psi_{1},\Psi_{2},...,\Psi_{k}$ be all the indecomposable $\cat{C}$-sets, and consider $\Omega_{n}$ the $\cat{D}$-sets defined for the quiver in the statements above (Example \ref{Omega_n} or Example  \ref{Omega_n2} respectively). Note that $\Omega_{n}$ is embedded in $ri(\Omega_n)$ and it is isomorphic to the restriction $ri(\Omega_n)|_{\cat{D}}$. As $ri(\Omega_n)$ is a $\cat{C}$-set, it follows that 
		\[
			ri(\Omega_n) \cong \Psi_{1}^{a_{1}} \sqcup \Psi_{2}^{a_{2}} \sqcup \cdots \sqcup \Psi_{k}^{a_{k}}.
		\]
		It also satisfies that
		\[
			\Omega_n \cong ri(\Omega_n)\vert_{\cat{D}} \cong \left(\Psi_{1}|_{\cat{D}} \right)^{a_{1}} \sqcup \left(\Psi_{2}|_{\cat{D}} \right)^{a_{2}} \sqcup \cdots \sqcup \left(\Psi_{k}|_{\cat{D}} \right)^{a_{k}}.
		\]
		
		As each $\Psi_{i}|_{\cat{D}}$ can be divided into indecomposable $\cat{D}$-sets, it follows that \-$\Psi_{i}|_{\cat{D}}= \bigsqcup_{j=1}^{m}\Psi_{ij}$, and with this, $\Omega_{n}$ is then embedded in $\bigsqcup_{i,j} (\Psi_{ij})^{a_{i}}$. In particular, as $\Omega_{n}$ is indecomposable, there exists a pair of indices $i_{n},j_{n}$ such that 
		\[
		\begin{tikzcd}
			\Omega_n \arrow[hook]{r} & \Psi_{i_n, j_n} \arrow[hook]{r} & \Psi_{i_n}.
		\end{tikzcd}
		\]
		
		Since the number of indecomposable $\cat{C}$-sets is finite, there must exist one whose restriction to $\cat{D}$ has maximal size. Let $m$ denote this maximal size. Then it follows that
		\[
		\begin{tikzcd}
			\Omega_{m + 1} \arrow[hook]{r} & \Psi_{i_{m + 1}},
		\end{tikzcd}
		\]
		which leads to a contradiction. Then there is no morphism in $\cat{C}$ between two objects that is not an isomorphism. Since the category $\cat{C}$ is equivalent to its skeleton --and the skeleton is a groupoid--, we conclude that $\cat{C}$ is a groupoid.

		\item[]  $iii) \implies vi)$
		To verify this point, it is enough to show that for every $x \in \obj{\cat{C}}$ and every $a \in \Omega(x)$, the equality $(\gen{\cat{C}}{\Omega}{a})(y) = (\orbit{\cat{C}}{\Omega}{a})(y)$ holds for every $y \in \obj{\cat{C}}$. The inclusion in one direction is clear, since $\gen{\cat{C}}{\Omega}{a}$ is a $\cat{C}$-subset of $\orbit{\cat{C}}{\Omega}{a}$. 
		
		Given an element $b \in (\orbit{\cat{C}}{\Omega}{a})(y)$, there exists a walk $\beta: x \rightsquigarrow y$ such that $b = \beta \cdot a$. Since $\cat{C}$ is a groupoid, $\beta$ can be considered just as a morphism in $\cat{C}$. It follows that $b = \beta \cdot a \in (\gen{\cat{C}}{\Omega}{a})(y)$. Therefore $(\orbit{\cat{C}}{\Omega}{a})(y) \subseteq (\gen{\cat{C}}{\Omega}{a})(y)$ and this happens for every $y \in \obj{\cat{C}}$. Hence, the equality holds.

		\item[]  $iv) \implies i)$  Let $\Omega$ be a $\cat{C}$-set. Then $\Omega = \bigsqcup_{i=1}^{n} \Omega_i$, where each $\Omega_i$ is an indecomposable $\cat{C}$-set. For every $x \in \obj{\cat{C}}$ and every element $u \in \Omega_i(x)$, we have $\Omega_i = \orbit{\cat{C}}{\Omega_i}{u}$. By hypothesis, $\Omega_i = \gen{\cat{C}}{\Omega_i}{u}$, for all $x \in \obj{\cat{C}}$ and all $u \in \Omega_i(x)$. Hence, by Proposition~\ref{simples}, $\Omega_i$ is a simple $\cat{C}$-set. This implies that $\cat{C}$ is a semisimple category. \qedhere
	\end{itemize}
\end{proof}

\begin{cor}
	Given a category $\cat{C}$, its Burnside ring has finite rank if and only if $\cat{C}$ is a groupoid. 
\end{cor}

This follows directly from the definition of the Burnside ring and the fact that the set of the isomorphism classes of the indecomposable objects forms a basis for the ring.

This work indicates that the Burnside ring of a finite category with infinite rank remains largely unexplored, with several fundamental questions still open. Among these are: What is the structure of its spectrum? What are the idempotent elements of the Burnside ring? How can simple $\cat{C}$-sets be characterized? It is also natural to ask whether the Burnside ring can characterize the category $\cat{C}$ in the general (non-groupoid) case. These and other questions point to a promising direction for further research.

\section*{Acknowledgements}

The authors would like to thank Peter Webb, who visited Mexico in 2023 to enthusiastically talk about his preprint \cite{Webb23} and encouraged us to continue working on these topics. We also thank Omar Antol\'in for his valuable comments in one of our examples.


\bibliographystyle{alpha}
\bibliography{bibliografia}

\end{document}